\documentclass{article}
\usepackage{amsmath}
\usepackage{amssymb}

\newtheorem{theorem}{Theorem}

\newtheorem{definition}[theorem]{Definition}

\newtheorem{lemma}[theorem]{Lemma}

\newtheorem{proposition}[theorem]{Proposition}
\newtheorem{remark}[theorem]{Remark}

\newenvironment{proof}[1][Proof]{\noindent\textbf{#1.} }{\ \rule{0.5em}{0.5em}}
\input{tcilatex}
\begin{document}

\title{Ultrafunctions and applications}
\author{Vieri Benci\thanks{
Dipartimento di Matematica Applicata, Universit\`{a} degli Studi di Pisa,
Via F. Buonarroti 1/c, Pisa, ITALY and Department of Mathematics, College of
Science, King Saud University, Riyadh, 11451, SAUDI ARABIA. e-mail: \texttt{%
benci@dma.unipi.it}}, Lorenzo Luperi Baglini}
\maketitle

\begin{abstract}
This paper deals with a new kind of generalized functions, called
"ultrafunctions" which have been introduced recently \cite{ultra} and
developed in \cite{belu2012} and \cite{belu2013}. Their peculiarity is that
they are based on a Non-Archimedean field namely on a field which contains
infinite and infinitesimal numbers. Ultrafunctions have been introduced to
provide generalized solutions to equations which do not have any solutions
not even among the distributions. Some of these applications will be
presented in the second part of this paper.

\medskip

\end{abstract}

\tableofcontents
\section{Introduction}

In many circumstances, the notion of function is not sufficient to the needs
of a theory and it is necessary to extend it. We can recall, for example,
the heuristic use of symbolic methods, called operational calculus. A basic
book on operational calculus was Oliver Heaviside's Electromagnetic Theory
of 1899 \cite{H99}. Since justifications were not rigorous, these methods
had a bad reputation among the pure mathematics. The professional
mathematicians accepted for the first time a notion of generalized function
with the introduction of the Lebesgue integral. An integrable function, in
Lebesgue's theory, is equivalent to any other which is the same almost
everywhere. This means that its value at a given point is meaningless and
the centrality of the concept of function was replaced by that of
"equivalence classes of functions". Further, very important steps in this
direction have been the introduction of the weak derivative and of the Dirac
Delta function. The theory of Dirac and the theory of weak derivatives where
unified by Schwartz in the beautiful theory of distributions (see e.g. \cite%
{Sw51} and \cite{Sw66}), also thanks to the previous work of Leray and
Sobolev. Among people working in partial differential equations, the theory
of Schwartz has been accepted as definitive (at least until now), but other
notions of generalized functions have been introduced by Colombeau \cite%
{col85} and Mikio Sato \cite{sa59}, \cite{sa60}.

This paper deals with a new kind of generalized functions, called
"ultrafunctions", which have been introduced recently in \cite{ultra} and
developed in \cite{belu2012} and \cite{belu2013}. The peculiarity of
ultrafunctions is that they are based on the Non-Archimedean Mathematics
(NAM). NAM is mathematics on fields which contain infinite and infinitesimal
numbers (Non-Archimedan fields). In the years around 1900, NAM was
investigated by prominent mathematicians such as Du Bois-Reymond \cite{DBR},
Veronese \cite{veronese}, David Hilbert \cite{hilb} and Tullio Levi-Civita 
\cite{LC}, but then it has been forgotten until the '60s when Abraham
Robinson presented his Non Standard Analysis (NSA) \cite{rob}. We refer to
Ehrlich \cite{el06} for a historical analysis of these facts and to Keisler 
\cite{keisler76} for a very clear exposition of NSA.

The ultrafunctions have been introduced to provide generalized solutions to
equations which do not have any solutions not even among the distributions.

The main peculiarities of ultrafunctions, as presented in this paper, are
the following:

\begin{itemize}
\item the ultrafunctions are functions which take their value in $\mathbb{R}%
^{\ast },$ which is a Non Archimedean field containing $\mathbb{R}$; this
field is a nonstandard extension of $\mathbb{R}$ (cf. e.g. \cite{keisler76});

\item any real integrable function $f:\Omega \rightarrow \mathbb{R\ (}\Omega
\subseteq \mathbb{R}^{N})$ can be extended to an ultrafunction 
\begin{equation*}
\widetilde{f}:\Omega ^{\mathbb{\ast }}\rightarrow \mathbb{R}^{\mathbb{\ast }%
},
\end{equation*}%
where $\Omega \subset \Omega ^{\mathbb{\ast }}\subset \left( \mathbb{R}^{%
\mathbb{\ast }}\right) ^{N};$

\item to any distribution $T\in \mathcal{D}^{\prime }\left( \Omega \right) $
we can associate an ultrafunction%
\begin{equation*}
\widetilde{T}:\Omega ^{\mathbb{\ast }}\rightarrow \mathbb{R}^{\mathbb{\ast }}
\end{equation*}%
such that, $\forall \varphi \in \mathcal{D}\left( \Omega \right) ,$%
\begin{equation*}
\left\langle T,\varphi \right\rangle =\int^{\ast }\widetilde{T}(x)\widetilde{%
\varphi }(x)dx,
\end{equation*}%
where $\int^{\ast }$ is a suitable extension of the definite integral to the
ultrafunctions;

\item any functional $J$ defined on a function space $V(\Omega )$ can be
extended to a functional $\widetilde{J}$ defined on a ultrafunction space $%
V_{\Lambda }(\Omega )\supset V(\Omega );$ moreover, if $J$ is coercive, $%
\widetilde{J}$ has a minimum in $V_{\Lambda }(\Omega );$

\item any linear or nonlinear differential operator $A(u)$ defined on a
suitable function space $V(\Omega )$ can be extended to an operator%
\begin{equation*}
\widetilde{A}:V_{\Lambda }(\Omega )\rightarrow V_{\Lambda }(\Omega )
\end{equation*}%
and the equation%
\begin{equation*}
\widetilde{A}(u)=\widetilde{f}
\end{equation*}%
might have a solution in $V_{\Lambda }(\Omega )$ even if the equation $%
A(u)=f $ has no solution in $V(\Omega );$

\item the main strategy to prove the existence of generalized solutions in
the space of ultrafunctions is relatively simple; it is just a variant of
the Faedo-Galerkin method.
\end{itemize}

The theory of ultrafunctions makes a large use of the techniques of NSA.
However, our approach to Non Archimedean Mathematics is quite different from
the spirit of Nonstandard Analysis: there are two main differences, one in
the aims and one in the methods.

Let us examine the difference in the aims. We think that infinitesimal and
infinite numbers should not be considered just as entities living in a
parallel universe (the nonstandard universe) which are only a tool to prove
some statement relative to our universe (the standard universe). On the
contrary, we think that they are mathematical entities which have the same
status of the others and can be used to build models as any other
mathematical entity. Actually, in our opinion, the advantages of a theory
which includes infinitesimals rely more on the possibility of making new
models rather than in the proving techniques. Our papers \cite{BGG} and \cite%
{nap}, as well as this one, are inspired by this principle. Actually, this point of view is not completely 
new: e.g., in the seventies and the first
years of the eighties some mathematical models involving infinitesimals for
economics and phisics have been constructed (see e.g. \cite{brown} and \cite%
{goze}). 

As far as the methods are concerned, we introduce a non-Archimedean field
via a new notion of limit (the $\Lambda $-limit, see section \ref{OL}). In
general, the $\Lambda $-limit of a sequence of "mathematical entities" is a
new object which preserves most of the properties of the "approximating
objects". This kind of new objects are called \textbf{internal} and, in
general, they are not present in the usual Mathematics. Infinite and
infinitesimal numbers, as well as the ultrafunctions, are internal objects.
Moreover, this notion of limit allows us to make a very limited use of the
formal logic: the Transfer Principle (or Leibnitz Principle) is given by Th. %
\ref{billo} and it is not necessary to introduce a formal language.\\
Moreover, another difference between our approach and the traditional one is that we do not assume the existence of two distinct mathematical universes. This idea is shared with the Nelson's approach to NSA called  \textbf{Internal Set Theory} (see \cite{nelson}). Nevertheless, our theory and IST have many differencies: e.g., in IST it is postulated the existence of 
infinitesimal elements in $\mathbb{R}$ itself, while we do not change the nature of this set. 

\subsection{Notations}

Let $\Omega $\ be a subset of $\mathbb{R}^{N};$ in concrete cases $\Omega $
will be an open set or the closure of an open set: then

\begin{itemize}
\item $\mathcal{F}\left( \Omega ,E\right) $ denotes the set of the functions
defined in $\Omega $ with values in $E;$

\item $\mathcal{F}\left( \Omega \right) $ denotes the set of the real
functions defined in $\Omega ;$

\item $\mathcal{C}\left( \overline{\Omega }\right) $ denotes the set of real
continuous functions defined on $\overline{\Omega };$

\item $\mathcal{C}_{0}\left( \overline{\Omega }\right) $ denotes the set of
real continuous functions on $\overline{\Omega }$ such that $u(x)=0$ on $%
\partial \Omega ;$

\item $\mathcal{C}^{k}\left( \Omega \right) $ denotes the set of real
functions defined on $\Omega \subset \mathbb{R}^{N}$ which have continuous
derivatives up to the order $k;$

\item $\mathcal{D}\left( \Omega \right) $ denotes the set of infinitely
differentiable real functions with compact support defined on $\Omega
\subset \mathbb{R}^{N};\ \mathcal{D}^{\prime }\left( \Omega \right) $
denotes the topological dual of $\mathcal{D}\left( \Omega \right) $, namely
the space of distributions on $\Omega ;$

\item $W^{1,p}(\Omega )$ is the usual Sobolev space defined as the set of
functions $u\in L^{p}\left( \Omega \right) $ such that $\nabla u\in
L^{p}\left( \Omega \right) ;$

\item $H^{1}(\Omega )=W^{1,2}(\Omega );$

\item $W_{0}^{1,p}(\Omega )$ is the closure of $\mathcal{D}\left( \Omega
\right) $ in $W^{1,p}(\Omega );$

\item $H_{0}^{1}(\Omega )=W_{0}^{1,2}(\Omega );$

\item $H^{-1}(\Omega )$ is the topological dual of $H_{0}^{1}(\Omega ).$
\end{itemize}

\section{$\Lambda $-theory\label{lt}}

In this section we present the basic notions of Non Archimedean Mathematics
and of Nonstandard Analysis following a method inspired by \cite{BDN2003}
(see also \cite{benci95}, \cite{benci99}, \cite{ultra} and \cite{belu2012}).

\subsection{Non Archimedean Fields\label{naf}}

Here, we recall the basic definitions and facts regarding Non Archimedean
fields. In the following, ${\mathbb{K}}$ will denote an ordered field. We
recall that such a field contains (a copy of) the rational numbers. Its
elements will be called numbers.

\begin{definition}
Let $\mathbb{K}$ be an ordered field. Let $\xi \in \mathbb{K}$. We say that:

\begin{itemize}
\item $\xi $ is infinitesimal if, for all positive $n\in \mathbb{N}$, $|\xi
|<\frac{1}{n}$;

\item $\xi $ is finite if there exists $n\in \mathbb{N}$ such that $|\xi |<n$%
;

\item $\xi $ is infinite if, for all $n\in \mathbb{N}$, $|\xi |>n$
(equivalently, if $\xi $ is not finite).
\end{itemize}
\end{definition}

\begin{definition}
An ordered field $\mathbb{K}$ is called Non-Archimedean if it contains an
infinitesimal $\xi \neq 0$.
\end{definition}

It's easily seen that all infinitesimal are finite, that the inverse of an
infinite number is a nonzero infinitesimal number, and that the inverse of a
nonzero infinitesimal number is infinite.

\begin{definition}
A superreal field is an ordered field $\mathbb{K}$ that properly extends $%
\mathbb{R}$.
\end{definition}

It is easy to show, due to the completeness of $\mathbb{R}$, that there are
nonzero infinitesimal numbers and infinite numbers in any superreal field.
Infinitesimal numbers can be used to formalize a new notion of "closeness":

\begin{definition}
\label{def infinite closeness} We say that two numbers $\xi ,\zeta \in {%
\mathbb{K}}$ are infinitely close if $\xi -\zeta $ is infinitesimal. In this
case we write $\xi \sim \zeta $.
\end{definition}

Clearly, the relation "$\sim $" of infinite closeness is an equivalence
relation.

\begin{theorem}
If $\mathbb{K}$ is a superreal field, every finite number $\xi \in \mathbb{K}
$ is infinitely close to a unique real number $r\sim \xi $, called the 
\textbf{shadow} or the \textbf{standard part} of $\xi $.
\end{theorem}

Given a finite number $\xi $, we denote its shadow as $sh(\xi )$, and we put 
$sh(\xi )=+\infty $ ($sh(\xi )=-\infty $) if $\xi \in \mathbb{K}$ is a
positive (negative) infinite number.

\subsection{The $\Lambda $-limit\label{OL}}

In this section we will introduce a superreal field $\mathbb{K}$ and we will
analyze its main properties by mean of the $\Lambda $-theory (for complete
proofs, the reader is referred to \cite{ultra}, \cite{belu2012}).

We set
\begin{equation*}
\mathfrak{L}=\mathcal{P}_{\omega }(\mathbb{R}^{N});
\end{equation*}%
where $\mathcal{P}_{\omega }(\mathbb{R}^{N})$ denotes the family of finite
subsets of $\mathbb{R}^{N}$. We will refer to $\mathfrak{L}$ as the
"parameter space". Clearly $\left( \mathfrak{L},\subset \right) $ is a
directed set. We recall that a directed set is a partially ordered set $%
(D,\prec )$ such that, $\forall a,b\in D,\ \exists c\in D$ such that 
\begin{equation*}
a\prec c\ \ \text{and}\ \ b\prec c.
\end{equation*}%
A function $\varphi :D\rightarrow E$ defined on a directed set will be
called \textit{net }(with values in $E$). A net $\varphi $ is the
generalization of the notion of sequence and it has been constructed in such
a way that the Weierstrass definition of limit makes sense: if $\varphi
_{\lambda }$ is a real net, we have that 
\begin{equation*}
\underset{\lambda \rightarrow \infty }{\lim }\varphi _{\lambda }=L
\end{equation*}%
if and only if 
\begin{equation}
\forall \varepsilon >0\text{ }\exists \lambda _{0}>0\text{\ such that, }%
\forall \lambda >\lambda _{0},\ \left\vert \varphi _{\lambda }-L\right\vert
<\varepsilon .  \label{limite}
\end{equation}

The key notion of the $\Lambda $-theory is the $\Lambda $-limit. Also the $%
\Lambda $-limit is defined for real nets but differs from the limit defined
by (\ref{limite}) mainly for the fact that there exists a Non Archimedean
field in which every real net admits a limit.

Now, we will present the notion of $\Lambda $-limit axiomatically:

\bigskip

{\Large Axioms of\ the }$\Lambda ${\Large -limit}

\begin{itemize}
\item \textsf{(}$\Lambda $-\textsf{1)}\ \textbf{Existence Axiom.}\ \textit{%
There is a superreal field} $\mathbb{K}\supset \mathbb{R}$ \textit{such that
every net }$\varphi :\mathfrak{L}\rightarrow \mathbb{R}$\textit{\ has a
unique limit }$L\in \mathbb{K}{\ }($\textit{called the} "$\Lambda $-limit" 
\textit{of}\emph{\ }$\varphi .)$ \textit{The} $\Lambda $-\textit{limit of }$%
\varphi $\textit{\ will be denoted as} 
\begin{equation*}
L=\lim_{\lambda \uparrow \Lambda }\varphi (\lambda ).
\end{equation*}%
\textit{Moreover we assume that every}\emph{\ }$\xi \in \mathbb{K}$\textit{\
is the }$\Lambda $-\textit{limit\ of some real function}\emph{\ }$\varphi :%
\mathfrak{L}\rightarrow \mathbb{R}$\emph{. }

\item ($\Lambda $-2)\ \textbf{Real numbers axiom}. \textit{If }$\varphi
(\lambda )$\textit{\ is} \textit{eventually} \textit{constant}, \textit{%
namely} $\exists \lambda _{0}\in \mathfrak{L},r\in \mathbb{R}$ such that $%
\forall \lambda \supset \lambda _{0},\ \varphi (\lambda )=r,$ \textit{then}%
\begin{equation*}
\lim_{\lambda \uparrow \Lambda }\varphi (\lambda )=r.
\end{equation*}

\item ($\Lambda $-3)\ \textbf{Sum and product Axiom}.\ \textit{For all }$%
\varphi ,\psi :\mathfrak{L}\rightarrow \mathbb{R}$\emph{: }%
\begin{eqnarray*}
\lim_{\lambda \uparrow \Lambda }\varphi (\lambda )+\lim_{\lambda \uparrow
\Lambda }\psi (\lambda ) &=&\lim_{\lambda \uparrow \Lambda }\left( \varphi
(\lambda )+\psi (\lambda )\right) ; \\
\lim_{\lambda \uparrow \Lambda }\varphi (\lambda )\cdot \lim_{\lambda
\uparrow \Lambda }\psi (\lambda ) &=&\lim_{\lambda \uparrow \Lambda }\left(
\varphi (\lambda )\cdot \psi (\lambda )\right) .
\end{eqnarray*}
\end{itemize}

The proof that this set of axioms $\{$($\Lambda $-1)\textsf{,}($\Lambda $%
-2),($\Lambda $-3)$\}$ is consistent can be found e.g. in \cite{ultra} or in 
\cite{belu2013}.

\subsection{Natural extension of sets and functions}

The notion of $\Lambda $-limit can be extended to sets and functions in the
following way:

\begin{definition}
Let $E_{\lambda },$ $\lambda \in \mathfrak{L},$ be a family of sets in $%
\mathbb{R}^{N}.$ We pose%
\begin{equation*}
\lim_{\lambda \uparrow \Lambda }\ E_{\lambda }:=\left\{ \lim_{\lambda
\uparrow \Lambda }\psi (\lambda )\ |\ \psi (\lambda )\in E_{\lambda
}\right\} .
\end{equation*}%
A set which is a $\Lambda $-\textit{limit\ is called \textbf{internal}.} In
particular, if $\forall \lambda \in \mathfrak{L,}$ $E_{\lambda }=E,$ we set $%
\lim_{\lambda \uparrow \Lambda }\ E_{\lambda }=E^{\ast },\ $namely%
\begin{equation*}
E^{\ast }:=\left\{ \lim_{\lambda \uparrow \Lambda }\psi (\lambda )\ |\ \psi
(\lambda )\in E\right\} .
\end{equation*}%
$E^{\ast }$ is called the \textbf{natural extension }of $E.$
\end{definition}

Notice that, while the $\Lambda $-limit\ of a constant sequence of numbers
gives this number itself, a constant sequence of sets gives a larger set,
namely $E^{\ast }$. In general, the inclusion $E\subseteq E^{\ast }$ is
proper.

This definition, combined with axiom ($\Lambda $-1$)$, entails that 
\begin{equation*}
\mathbb{K}=\mathbb{R}^{\ast }.
\end{equation*}

Given any set $E,$ we can associate to it two sets: its natural extension $%
E^{\ast }$ and the set $E^{\sigma },$ where%
\begin{equation}
E^{\sigma }=\left\{ x^{\ast }\ |\ x\in E\right\} .  \label{sigmaS}
\end{equation}

Clearly $E^{\sigma }$ is a copy of $E;$ however it might be different as set
since, in general, $x^{\ast }\neq x.$ Moreover $E^{\sigma }\subset E^{\ast }$
since every element of $E^{\sigma }$ can be regarded as the $\Lambda $%
-limit\ of a constant sequence.

\begin{definition}
Let 
\begin{equation*}
f_{\lambda }:\ E_{\lambda }\rightarrow \mathbb{R},\ \ \lambda \in \mathfrak{L%
},
\end{equation*}%
be a family of functions. We define a function%
\begin{equation*}
f:\left( \lim_{\lambda \uparrow \Lambda }\ E_{\lambda }\right) \rightarrow 
\mathbb{R}^{\ast }
\end{equation*}%
as follows: for every $\xi \in \left( \lim_{\lambda \uparrow \Lambda }\
E_{\lambda }\right) $ we pose%
\begin{equation*}
f\left( \xi \right) :=\lim_{\lambda \uparrow \Lambda }\ f_{\lambda }\left(
\psi (\lambda )\right) ,
\end{equation*}%
where $\psi (\lambda )$ is a net of numbers such that 
\begin{equation*}
\psi (\lambda )\in E_{\lambda }\ \ \text{and}\ \ \lim_{\lambda \uparrow
\Lambda }\psi (\lambda )=\xi .
\end{equation*}%
A function which is a $\Lambda $-\textit{limit\ is called \textbf{internal}.}
In particular if, $\forall \lambda \in \mathfrak{L,}$ 
\begin{equation*}
f_{\lambda }=f,\ \ \ \ f:\ E\rightarrow \mathbb{R},
\end{equation*}%
we set 
\begin{equation*}
f^{\ast }=\lim_{\lambda \uparrow \Lambda }\ f_{\lambda }.
\end{equation*}%
$f^{\ast }:E^{\ast }\rightarrow \mathbb{R}^{\ast }$ is called the \textbf{%
natural extension }of $f.$
\end{definition}

\subsection{Hyperfinite sets and hyperfinite sums\label{HE}}

\begin{definition}
An internal set is called \textbf{hyperfinite} if it is the $\Lambda $-limit
of a net $\varphi :\mathfrak{L}\rightarrow \mathfrak{Fin}$ where $\mathfrak{%
Fin}$ is a family of finite sets.
\end{definition}

For example the set 
\begin{equation*}
\lim_{\lambda \uparrow \Lambda }\lambda
\end{equation*}%
is hyperfinite; this set will be denoted by $\Lambda .$ By its definition we
have that%
\begin{equation*}
\Lambda =\left\{ \lim_{\lambda \uparrow \Lambda }\ x_{\lambda }\ |\
x_{\lambda }\in \lambda \right\} .
\end{equation*}

It is possible to add the elements of an hyperfinite set of numbers (or
vectors) as follows: let%
\begin{equation*}
A:=\ \lim_{\lambda \uparrow \Lambda }A_{\lambda }
\end{equation*}%
be an hyperfinite set of numbers (or vectors); then the hyperfinite sum of
the elements of $A$ is defined in the following way: 
\begin{equation*}
\sum_{a\in A}a=\ \lim_{\lambda \uparrow \Lambda }\sum_{a\in A_{\lambda }}a.
\end{equation*}%
In particular, if $A_{\lambda }=\left\{ a_{1}(\lambda ),...,a_{\beta
(\lambda )}(\lambda )\right\} \ $with\ $\beta (\lambda )\in \mathbb{N},\ $%
then setting 
\begin{equation*}
\beta =\ \lim_{\lambda \uparrow \Lambda }\ \beta (\lambda )\in \mathbb{N}%
^{\ast }
\end{equation*}%
we use the notation%
\begin{equation*}
\sum_{j=1}^{\beta }a_{j}=\ \lim_{\lambda \uparrow \Lambda }\sum_{j=1}^{\beta
(\lambda )}a_{j}(\lambda ).
\end{equation*}

\subsection{Qualified sets\label{qs}}

As one can expect, if two nets $\varphi ,\psi $ are equal on a "qualified"
subset of $\mathfrak{L}$ then they share the same $\Lambda $-limit. The
notion of "qualified" subset of $\mathfrak{L}$ can be precisely defined in
the following.

If $\mathfrak{Q}\subset \mathfrak{L}$ and $\varphi :\mathfrak{Q}\rightarrow
E $, the following notation is quite useful:%
\begin{equation*}
\Lambda \text{-}\lim_{\lambda \in \mathfrak{Q}}\varphi (\lambda
):=\lim_{\lambda \uparrow \Lambda }\widetilde{\varphi }(\lambda )
\end{equation*}%
where 
\begin{equation*}
\widetilde{\varphi }(\lambda )=\left\{ 
\begin{array}{cc}
\varphi (\lambda ) & \text{for}\ \ \lambda \in \mathfrak{Q}, \\ 
\emptyset & \text{for\ }\ \lambda \notin \mathfrak{Q}.%
\end{array}%
\right.
\end{equation*}%
Clearly, taking $\mathfrak{Q}=\mathfrak{L}$ , we have that%
\begin{equation*}
\Lambda \text{-}\lim_{\lambda \in \mathfrak{L}}\varphi (\lambda
)=\lim_{\lambda \uparrow \Lambda }\varphi (\lambda ).
\end{equation*}

In general, it is not difficult to prove that, for any set $\mathfrak{Q}%
\subset \mathfrak{L},$ it can occur only one of the two following
possibilities:%
\begin{eqnarray*}
\Lambda \text{-}\lim_{\lambda \in \mathfrak{Q}}\varphi (\lambda )
&=&\lim_{\lambda \uparrow \Lambda }\varphi (\lambda )\text{, or} \\
\Lambda \text{-}\lim_{\lambda \in \mathfrak{Q}}\varphi (\lambda )
&=&\emptyset .
\end{eqnarray*}

We use this notation to introduce the notion of qualified set:

\begin{definition}
\label{qua}We say that a set $\mathfrak{Q}\subset \mathfrak{L}$ is qualified
if, for every net $\varphi ,$ we have that 
\begin{equation*}
\Lambda \text{-}\lim_{\lambda \in \mathfrak{Q}}\varphi (\lambda
)=\lim_{\lambda \uparrow \Lambda }\varphi (\lambda ).
\end{equation*}
\end{definition}

By the above definition we have that the $\Lambda $-limit of a net $\varphi $
depends only on the values that $\varphi $ takes on a qualified set. It is
easy to see that (nontrivial) qualified sets exist. For example, by ($%
\Lambda $-2), we can deduce that, for every $\lambda _{0}\in \Lambda ,$ the
set%
\begin{equation*}
Q\left( \lambda _{0}\right) :=\left\{ \lambda \in \Lambda \ |\ \lambda
_{0}\subseteq \lambda \right\}
\end{equation*}%
is qualified. The family of qualified sets $\mathcal{U}$ satisfies the
following assumptions.

\begin{proposition}
\label{quaqua}The family of qualified sets forms a non principal ultrafilter 
$\mathcal{U}$, namely:

\begin{enumerate}
\item if $\mathfrak{Q}\in \mathcal{U}$ and $\mathfrak{Q}\subset \mathfrak{R,}
$ then $\mathfrak{R}\in \mathcal{U}$;

\item if $\mathfrak{Q}$ and $\mathfrak{R}\in \mathfrak{\mathcal{U}}$, then $%
\mathfrak{Q}\cap \mathfrak{R}\in \mathcal{U}$;

\item if $\mathfrak{Q}\in \mathfrak{\mathcal{U}}$, then $\mathfrak{Q}$ is
infinite;

\item $\mathfrak{Q}\in \mathfrak{\mathcal{U}}$ if and only if $\mathfrak{L}%
\backslash \mathfrak{Q}\notin \mathfrak{\mathcal{U}}$.
\end{enumerate}
\end{proposition}

The notion of qualified set allows to state Theorem \ref{billo}, which is a
sort of "weak form" of the Transfer (or Leibnitz) Principle (see e.g. \cite%
{keisler76}):

\begin{theorem}
\label{billo}\textbf{(Transfer Principle - weak form)} Let $\mathcal{R}$ be
a relation and let $\varphi $, $\psi $ be two nets. Then the following
statements are equivalent:

\begin{itemize}
\item there exists a qualified set $\mathfrak{Q}$ such that 
\begin{equation*}
\forall \lambda \in \mathfrak{Q},\ \varphi (\lambda )\mathcal{R}\psi
(\lambda );
\end{equation*}

\item we have 
\begin{equation*}
\left( \underset{\lambda \uparrow \Lambda }{\lim }\varphi (\lambda )\right) 
\mathcal{R}^{\ast }\left( \underset{\lambda \uparrow \Lambda }{\lim }\psi
(\lambda )\right) .
\end{equation*}
\end{itemize}
\end{theorem}

\begin{proof} It is an immediate consequence of the definition of
qualified set.\end{proof}

\section{Ultrafunctions}

In this section, we will recall the notion of ultrafunction and we will
analyze its first properties.

\subsection{Definition of Ultrafunctions}

Let $\Omega $ be a set in $\mathbb{R}^{N}$ and let $V(\Omega )\ $be a
function vector space such that $\mathcal{D}(\Omega )\subseteq V(\Omega
)\subseteq \mathcal{C}^{0}(\overline{\Omega })\cap L^{1}(\Omega )\cap
L^{2}(\Omega ).$ We denote by 
\begin{equation*}
\left\{ e_{a}\right\} _{a\in \Omega }
\end{equation*}%
a Hamel basis of $V(\Omega ).$ We recall that a Hamel basis of $V(\Omega )$
has the continuoum cardinality and hence we can use the points of $\Omega $
as indices for this basis. For any $\lambda \in \mathfrak{L},$ we set%
\begin{equation*}
V_{\lambda }(\Omega )=Span\left\{ e_{a}\ |\ a\in \lambda \right\} .
\end{equation*}

\begin{definition}
Given the function space $V(\Omega )$ we set%
\begin{equation*}
V_{\Lambda }(\Omega ):=\lim_{\lambda \uparrow \Lambda }V_{\lambda }(\Omega ).
\end{equation*}%
$V_{\Lambda }(\Omega )$ will be called the \textbf{space of ultrafunctions}
generated by $V(\Omega ).$
\end{definition}

\begin{remark}
Sometimes, for grafic reasons, we will write $\left[ V(\Omega )\right]
_{\Lambda }$ instead of $V_{\Lambda }(\Omega );$ for example, if $V(\Omega )=%
\mathcal{C}^{1}(\Omega )\cap \mathcal{C}^{0}(\overline{\Omega })\cap
L^{1}(\Omega )\cap L^{2}(\Omega ),$ it makes sense to write%
\begin{equation*}
V_{\Lambda }(\Omega )=\left[ \mathcal{C}^{1}(\Omega )\cap \mathcal{C}^{0}(%
\overline{\Omega })\cap L^{1}(\Omega )\cap L^{2}(\Omega )\right] _{\Lambda }.
\end{equation*}%
In the applications, $V_{\Lambda }^{k}(\Omega )$ (or simply $V^{k}(\Omega )$%
) will denote the space of ultrafunctions generated by $\mathcal{C}%
^{k}(\Omega )\cap \mathcal{C}^{0}(\overline{\Omega })\cap L^{1}(\Omega )\cap
L^{2}(\Omega ),$ namely%
\begin{equation*}
V_{\Lambda }^{k}(\Omega )=\left[ \mathcal{C}^{k}(\Omega )\cap \mathcal{C}%
^{0}(\overline{\Omega })\cap L^{1}(\Omega )\cap L^{2}(\Omega )\right]
_{\Lambda }.
\end{equation*}
\end{remark}

So, given any vector space of functions $V(\Omega )$, we have the following
three properties:

\begin{enumerate}
\item the ultrafunctions in $V_{\Lambda }(\Omega )$ are $\Lambda $-limits of
functions in $V_{\lambda };$

\item the space of ultrafunctions $V_{\Lambda }(\Omega )$ is a vector space
of hyperfinite dimension, since it is a $\Lambda $-limit of a net of finite
dimensional vector spaces;

\item $V_{\Lambda }(\Omega )$ includes $V(\Omega ).$
\end{enumerate}

Hence the ultrafunctions are particular internal functions 
\begin{equation*}
u:\Omega {^{\ast }}\rightarrow {\mathbb{R}^{\ast }.}
\end{equation*}

By definition, the dimension of $V_{\Lambda }(\Omega )$ (that we denote by $%
\mathfrak{n})$ is equal to the internal cardinality of any of its bases, and
the following formula holds: 
\begin{equation}
\mathfrak{n}=\lim_{\lambda \uparrow \Lambda }\text{dim}(V_{\lambda }(\Omega
))=\lim_{\lambda \uparrow \Lambda }card\left( \lambda \right) =card^{\ast
}\left( \Lambda \right) .  \label{anna}
\end{equation}

\begin{remark}
\label{nina}Notice that the natural extension $f^{\ast }$ of a function $f$
is an ultrafunction if and only if $f\in V(\Omega ).$
\end{remark}

\begin{proof} Let $f\in V(\Omega ).$ Then, eventually, $f\in V_{\lambda }$
and hence 
\begin{equation*}
f^{\ast }=\lim_{\lambda \uparrow \Lambda }f\in \lim_{\lambda \uparrow
\Lambda }\ V_{\lambda }(\Omega )=V_{\Lambda }(\Omega ).
\end{equation*}

Conversely, if $f\notin V(\Omega )$ then by the Transfer Principle (Th. \ref%
{billo}) it follows that $f^{\ast }\notin V^{\ast }(\Omega )$ and, since $%
V_{\Lambda }(\Omega )\subset V^{\ast }(\Omega )$, this entails the thesis.\end{proof}

\bigskip

Since $V_{\Lambda }(\Omega )\subset \left[ L^{2}(\mathbb{R})\right] ^{\ast
}, $ we can equip $V_{\Lambda }(\Omega )$ with the following scalar product:%
\begin{equation*}
\left( u,v\right) =\int^{\ast }u(x)v(x)\ dx,
\end{equation*}%
where $\int^{\ast }$ is the natural extension of the Lebesgue integral
considered as a functional%
\begin{equation*}
\int :L^{1}(\Omega )\rightarrow {\mathbb{R}}.
\end{equation*}%
The norm of an ultrafunction will be given by 
\begin{equation*}
\left\Vert u\right\Vert =\left( \int^{\ast }|u(x)|^{2}\ dx\right) ^{\frac{1}{%
2}}.
\end{equation*}

\subsection{Delta and Sigma Basis}

In this section we introduce two particular bases for $V_{\Lambda }\left(
\Omega \right) $ and we study their main properties. We start by defining
the \textit{Delta ultrafunctions}:

\begin{definition}
\label{dede}Given a point $q\in \overline{\Omega }^{\ast },$ we denote by $%
\delta _{q}(x)$ an ultrafunction in $V_{\Lambda }\left( \Omega \right) $
such that 
\begin{equation}
\forall v\in V_{\Lambda }(\Omega ),\ \int^{\ast }v(x)\delta _{q}(x)dx=v(q).
\label{deltafunction}
\end{equation}%
$\delta _{q}(x)$ is called Delta (or the Dirac) ultrafunction centered in $q$%
.
\end{definition}

Let us see the main properties of the Delta ultrafunctions:

\begin{theorem}
\label{delta} We have the following properties:

\begin{enumerate}
\item For every $q\in \Omega ^{\ast }$ there exists an unique Delta
ultrafunction centered in $q;$

\item for every $a,\ b\in \Omega ^{\ast }\ \delta _{a}(b)=\delta _{b}(a);$

\item $\left\Vert \delta _{q}\right\Vert ^{2}=\delta _{q}(q);$

\item if $\left\{ e_{j}\right\} _{j\in J}$ is an orthonomal basis of $%
V_{\Lambda }(\Omega )$ then, for every $q\in \overline{\Omega }^{\ast }$ 
\begin{equation*}
\delta _{q}(x)=\sum_{j\in J}e_{j}(q)e_{j}(x).
\end{equation*}
\end{enumerate}
\end{theorem}

\begin{proof} See \cite{belu2013}.\end{proof}

\bigskip

Now we will recall some basic facts of linear algebra which will be used
later. Given a basis $\left\{ e_{j}\right\} $ in a finite dimensional vector
space $V,$ the dual basis of $\left\{ e_{j}\right\} $ is the basis $\left\{
e_{j}^{\prime }\right\} $ of the dual space $V^{\prime }$ defined by the
following relation:%
\begin{equation*}
e_{j}^{\prime }\left[ e_{k}\right] =\delta _{jk}.
\end{equation*}%
If $V$ has a scalar product$\ \left( \cdot \ |\ \cdot \right) $ then $V$ and 
$V^{\prime }$ can be identified and hence the dual basis $\left\{
e_{j}^{\prime }\right\} $ is characterized by the following relation:%
\begin{equation*}
\left( e_{j}^{\prime }\ |\ e_{k}\right) =\delta _{jk}.
\end{equation*}

The notion of dual basis allows to give the following definition:

\begin{definition}
A Delta-basis $\left\{ \delta _{a}(x)\right\} _{a\in \Sigma }$ $(\Sigma
\subset \Omega ^{\ast })$ is a basis for $V_{\Lambda }(\Omega )$ whose
elements are Delta ultrafunctions. Its dual basis $\left\{ \sigma
_{a}(x)\right\} _{a\in \Sigma }$ is called Sigma-basis. The set $\Sigma
\subset \Omega ^{\ast }$ is called set of independent points.
\end{definition}

So a Sigma-basis is characterized by the fact that, $\forall a,b\in \Sigma ,$%
\begin{equation}
\int^{\ast }\delta _{a}(x)\sigma _{b}(x)dx=\delta _{ab}.  \label{mimma}
\end{equation}

It is not difficult to prove the existence of a Delta-basis (\cite{belu2013}%
). We will list some properties of Delta- and Sigma-bases (for the proof see 
\cite{belu2013}):

\begin{theorem}
\label{tbase}A Delta-basis $\left\{ \delta _{q}(x)\right\} _{q\in \Sigma }$
and its dual basis $\left\{ \sigma _{q}(x)\right\} _{q\in \Sigma }$ satisfy
the following properties:

\begin{enumerate}
\item if $u\in V_{\Lambda }(\Omega )$, then%
\begin{equation*}
u(x)=\sum_{q\in \Sigma }\left( \int^{\ast }\sigma _{q}(\xi )u(\xi )d\xi
\right) \delta _{q}(x);
\end{equation*}

\item if $u\in V_{\Lambda }(\Omega )$, then%
\begin{equation}
u(x)=\sum_{q\in \Sigma }u(q)\sigma _{q}(x);  \label{brava+}
\end{equation}

\item if two ultrafunctions $u$ and $v$ coincide on a set of independent
points then they are equal;

\item if $\Sigma $ is a set of independent points and $a,b\in \Sigma $ then $%
\sigma _{a}(b)=\delta _{ab};$

\item for any $q\in \Omega ^{\ast }$ $\sigma _{q}(x)$ is well defined.
\end{enumerate}
\end{theorem}

\subsection{Extensions of functions, functionals and operators\label{effo}}

A measurable function $f$ can be identified with an element of the dual
space of $V(\Omega )$ provided that, $\forall v\in V(\Omega ),\ fv\ $is
integrable. In this case, we will write $f\in V^{\prime }(\Omega ).$ Every
function $f\in V^{\prime }(\Omega )$ can be extended to an ultrafunction $%
\widetilde{f}\in V_{\Lambda }(\Omega )$ just setting 
\begin{equation}
\widetilde{f}\left( x\right) =\sum_{a\in \Sigma }\left( \int^{\ast }f^{\ast
}\delta _{a}dx\right) \sigma _{a}(x).  \label{gianna}
\end{equation}%
The integral $\int^{\ast }f^{\ast }\delta _{a}dx$ makes sense since $\delta
_{a}\in V_{\Lambda }(\Omega )\subset V(\Omega )^{\ast }$ and $f^{\ast }\in
V^{\prime }(\Omega )^{\ast }.$ Notice that in general $\widetilde{f}\left(
x\right) \neq f^{\ast }(x);$ actually, the equality holds if and only if $%
f\in V(\Omega ).$

\bigskip

\textbf{Example: }If $\Omega $ is bounded then every function in $%
L^{1}(\Omega )$ can be extended to an ultrafunction $\widetilde{f}\in
V_{\Lambda }(\Omega ),$ since $L^{1}(\Omega )\subset \mathcal{C}^{\prime
}(\Omega )\subset V^{\prime }(\Omega ).$

\begin{remark}
If $f\in L^{2}(\Omega )$ then $\widetilde{f}$ is nothing else but the
orthogonal projection of $f^{\ast }$ on $V_{\Lambda }(\Omega ).$ More in
general, $\widetilde{f}\left( x\right) $ is the only function in $V_{\Lambda
}(\Omega )$ such that%
\begin{equation*}
\forall v\in V(\Omega ),\ \int^{\ast }\widetilde{f}(x)v(x)dx=\int^{\ast
}f(x)v(x)dx.
\end{equation*}
\end{remark}

A discussion on the previous remark can be found in \cite{belu2013}.

\begin{remark}
The formula (\ref{gianna}) is not the only way to identify an ultrafunction
with a standard function: an other possible way is the following:%
\begin{equation}
f\mapsto \sum_{a\in \Sigma }f(a)\sigma _{a}(x)  \label{giannina}
\end{equation}%
Notice that (\ref{gianna}) and (\ref{giannina}) are equal if and only if $%
f\in V(\Omega ).$ The identification (\ref{giannina}) is studied in detail
in \cite{algebra}, where it is used to construct an algebra of
ultrafunctions with good properties of coherence w.r.t. distributions.
However, for the applications presented here, the identification (\ref%
{gianna}) seems better.
\end{remark}

If%
\begin{equation*}
J:V(\Omega )\rightarrow \mathbb{R}
\end{equation*}%
is a functional then the restriction of $J^{\ast }$ to $V_{\Lambda }(\Omega
) $ is well defined and we will denote it by%
\begin{equation*}
\widetilde{J}:V_{\Lambda }(\Omega )\rightarrow \mathbb{R}^{\ast }.
\end{equation*}%
Now let 
\begin{equation*}
A:V(\Omega )\rightarrow W(\Omega )
\end{equation*}%
be an operator between function spaces. If $W(\Omega )\subset V^{\prime
}(\Omega )^{\ast },$ it is possible to extend this operator to an operator
between ultrafunctions%
\begin{equation*}
\widetilde{A}:V_{\Lambda }(\Omega )\rightarrow V_{\Lambda }(\Omega ),
\end{equation*}%
defining $\widetilde{A}(u)$ as the only ultrafunction such that, $\forall
\varphi \in V_{\Lambda }(\Omega ),$%
\begin{equation*}
\int^{\ast }\widetilde{A}(u)\varphi \ dx=\int^{\ast }A^{\ast }(u)\varphi \
dx.
\end{equation*}%
Notice that, by definition, if $u\in V(\Omega )^{\sigma }$ then $\widetilde{A%
}(u)=\widetilde{A(u)}.$

For example, if $V(\Omega )\subset \mathcal{C}^{1}(\left[ a,b\right] )$ and $%
\partial $ is the usual derivative, we can define the "generalized"
derivative $D:V_{\Lambda }(\Omega )\rightarrow V_{\Lambda }(\Omega )$ as
follows: $\forall \varphi \in V_{\Lambda }(\Omega )$ we pose%
\begin{equation*}
\int^{\ast }Du\varphi \ dx=\int^{\ast }\left( \partial u\right) ^{\ast
}\varphi dx=\int^{\ast }\left( u^{\prime }\right) ^{\ast }\varphi dx.
\end{equation*}%
It is easy to check that $D$ is the only operator on $\mathcal{C}^{1}(\left[
a,b\right] )$ such that, for every $u,v\in \left[ \mathcal{C}^{1}(\left[ a,b%
\right] )\right] _{\Lambda },$ we have%
\begin{equation}
\int^{\ast }Du\varphi \ dx=\left[ u\varphi \right] _{a}^{b}-\int^{\ast
}uD\varphi \ dx.  \label{manola}
\end{equation}%
Let us finally observe that if both $f,\partial f\in V(\Omega )$ then%
\begin{equation}
Df=\left( \partial f\right) ^{\ast }.  \label{zumpappero}
\end{equation}%
Notice that, from now on, $D$ will denote the ultrafunction derivative,
while we will denote by $\partial $ the usual derivative or the weak
derivative in the sense of distributions.

\subsection{Distributions}

For simplicity here we deal only with distributions in $\mathcal{D}^{\prime
}\left( \mathbb{R}\right) .$ It is well known that a distribution $T\in 
\mathcal{D}^{\prime }\left( \mathbb{R}\right) $ has the following
representation\footnote{%
See e.g. Rudin, Functional Analysis, Th. 6.28, pag.169}: 
\begin{equation}
T=\sum_{k=0}^{\infty }\partial ^{k}f_{k},  \label{mary}
\end{equation}%
where $f_{k}\in \mathcal{C}^{1}(\mathbb{R})$ and the sum is locally finite%
\footnote{%
Actually the formula below holds also for $f_{k}\in \mathcal{C}^{0}\left( 
\mathbb{R}\right) ;$ we have taken the $f_{k}\ $in $\mathcal{C}^{1}\left( 
\mathbb{R}\right) $ in order to give sense to Def. \ref{peppa}.}; namely,
for every $\varphi \in \mathcal{D}\left( \left[ a,b\right] \right) ,$ we
have that%
\begin{equation*}
\left\langle T,\varphi \right\rangle =\sum_{k=0}^{N(a,b)}\left( -1\right)
^{k}\int f(x)\partial ^{k}\varphi (x)dx,
\end{equation*}%
where $N(a,b)$ is a natural number which depends on $a$,$b$ and $T$. If $T$
is given by (\ref{mary}), we denote by $F_{T}$ the set%
\begin{equation*}
F_{T}=\{f_{k}\mid k\in \mathbb{N\}}
\end{equation*}%
and, if $N\in \mathbb{N}^{\ast },$ we pose 
\begin{equation*}
\{f_{1},...,f_{N}\}=\{f_{i}\in F_{T}^{\ast }\mid 0\leq k\leq N\}.
\end{equation*}%
Moreover, whenever we have $f\in V_{\Lambda }^{\prime }(\mathbb{R})$ (i.e. $f
$ such that $\int^{\ast }f(x)u(x)dx$ is well-posed for every $u\in
V_{\Lambda }(\mathbb{R})),$ we let $\widetilde{f}$ be the unique
ultrafunction such that, for every ultrafunction $v\in V_{\Lambda }(\mathbb{R%
}),$ we have%
\begin{equation*}
\int^{\ast }\widetilde{f}(x)v(x)dx=\int^{\ast }f(x)v(x)dx.
\end{equation*}%
Equivalently, if $\{\delta _{a}\}_{a\in \Sigma }$ is a delta basis of $%
V_{\Lambda }(\mathbb{R}),$ we have%
\begin{equation}
\widetilde{f}(x)=\sum_{a\in \Sigma }\left[ \int^{\ast }f(\xi )\delta
_{a}(\xi )d\xi \right] \sigma _{a}(x).  \label{bertoldo}
\end{equation}%
We still use the notation with $\widetilde{\cdot }$ because the association
given by (\ref{bertoldo}) is nothing more than the extension to $V_{\Lambda
}^{\prime }(\mathbb{R})$ of the association given by $\left( \ref{gianna}%
\right) $ (in the sense that, if $f\in V^{\prime }(\mathbb{R}),$ then $%
\widetilde{f}=\widetilde{f^{\ast }}$).

Let $\beta \ $be a fixed positive infinite number.

\begin{definition}
\label{peppa}We say that $\widetilde{T}$ is the ultrafunction which extends
the distribution (\ref{mary}) if 
\begin{equation}
\widetilde{T}=\sum_{k=0}^{N^{\ast }(-\beta ,\beta )}D^{k}\widetilde{f_{k}}.
\label{jane}
\end{equation}
\end{definition}

Let us note that the above definition is well posed, since 
\begin{equation*}
\{f_{k}\in F_{T}^{\ast }\mid 0\leq k\leq N^{\ast }(-\beta ,\beta
)\}\subseteq V_{\Lambda }^{\prime }(\mathbb{R}),
\end{equation*}%
and that it is justified by the following proposition:

\begin{proposition}
\label{barbapapa}$\forall \varphi \in \mathcal{D}(\mathbb{R}),$ $\forall
T\in \mathcal{D}^{\prime }(\mathbb{R})$ we have%
\begin{equation}
\int^{\ast }\widetilde{T}(x)\widetilde{\varphi }(x)dx=\left\langle T,\varphi
\right\rangle .  \label{katy}
\end{equation}
\end{proposition}

To prove Proposition \ref{barbapapa} we make use of the following lemma:

\begin{lemma}
\label{derivatona}For every $k\in \mathbb{N}^{\ast },$ for every $u\in
V_{\Lambda }(\mathbb{R}),$ for every $\varphi \in \mathcal{D}(\mathbb{R)}$
we have the following: 
\begin{equation*}
\int^{\ast }D^{k}u(x)\cdot \varphi ^{\ast }(x)dx=(-1)^{k}\int^{\ast
}u(x)\partial ^{k}\varphi ^{\ast }(x)dx.
\end{equation*}
\end{lemma}

\begin{proof} Let $a,b\in \mathbb{R}$ such that $\varphi \in \mathcal{D}(%
\left[ a,b\right] \mathbb{)}.$ We work by internal induction on $k$: if $k=0$
there is nothing to prove. Let us suppose the statement true for $k$. Then,
by $\left( \ref{manola}\right) ,$ we have%
\begin{equation*}
\int^{\ast }D^{k+1}(u(x))\varphi ^{\ast }(x)dx=\int^{\ast
}D(D^{k}(u(x)))\varphi ^{\ast }(x)dx=
\end{equation*}%
\begin{equation*}
-\int^{\ast }D^{k}(u(x))D\varphi ^{\ast }(x)dx+\left[ D^{k}u\cdot \varphi
^{\ast }\right] _{a}^{b}.
\end{equation*}

Since $\varphi \in \mathcal{D}([a,b]\mathbb{)}$ we have $\left[ D^{k}u\cdot
\varphi ^{\ast }\right] _{-\beta }^{\beta }=0.$ Moreover, since $\varphi \in 
\mathcal{D}$ then (by (\ref{zumpappero})) we have $D\varphi ^{\ast }=\left(
\partial \varphi \right) ^{\ast }\in \mathcal{D}([a,b]\mathbb{)}^{\ast }$.
So by the induction hypothesis we have%
\begin{equation*}
-\int^{\ast }D^{k}(u(x))D\varphi ^{\ast }(x)dx=-\int^{\ast
}D^{k}(u(x))\partial \varphi ^{\ast }(x)dx=
\end{equation*}%
\begin{equation*}
(-1)^{k+1}\int^{\ast }u(x)\partial ^{k+1}\varphi ^{\ast }(x)dx,
\end{equation*}%
and the thesis is proved.\end{proof}

\bigskip

\begin{proof} Proof of Proposition \ref{barbapapa}: Let us suppose that supp$%
(\varphi )\subseteq \lbrack a,b],$ where $a,b\in \mathbb{R}$ and supp$%
(\varphi ),$ as usual, is the support of $\varphi .$ By $\left( \ref{jane}%
\right) $ we have that 
\begin{equation*}
\int^{\ast }\widetilde{T}(x)\varphi ^{\ast }(x)dx=\sum_{k=0}^{N^{\ast
}(-\beta ,\beta )}\int^{\ast }D^{k}\widetilde{f_{k}}(x)\widetilde{\varphi }%
(x)dx.
\end{equation*}%
As a consequence of Lemma \ref{derivatona}, we obtain that%
\begin{equation*}
\sum_{k=0}^{N^{\ast }(-\beta ,\beta )}\int^{\ast }D^{k}\widetilde{f_{k}}(x)%
\widetilde{\varphi }(x)dx=\sum_{k=0}^{N^{\ast }(-\beta ,\beta )}\left(
-1\right) ^{k}\int^{\ast }f_{k}\partial ^{k}\varphi (x)dx,
\end{equation*}%
and since supp$(\varphi )\subseteq \lbrack a,b]\subset \lbrack -\beta ,\beta
]$ we get%
\begin{equation*}
\sum_{k=0}^{N^{\ast }(-\beta ,\beta )}\left( -1\right) ^{k}\int^{\ast
}f_{k}\partial ^{k}\varphi (x)dx=\sum_{k=0}^{N^{\ast }(a,b)}\left( -1\right)
^{k}\int^{\ast }f_{k}\partial ^{k}\varphi (x)dx=
\end{equation*}%
\begin{equation*}
\sum_{k=0}^{N(a,b)}\left( -1\right) ^{k}\int f_{k}\partial ^{k}\varphi
(x)dx=\left\langle T,\varphi \right\rangle .
\end{equation*}\end{proof}

\bigskip

Notice that (\ref{katy}) is not sufficient to characterize the ultrafuction $%
\widetilde{T};$ namely there are infinitely many ultrafunctions $u$ such
that $\int^{\ast }u\widetilde{\varphi }dx=\left\langle T,\varphi
\right\rangle $ \ for all $\varphi \in \mathcal{D}(\mathbb{R})$ (for a proof
of this fact, see \cite{belu2013}, Proposition 30).

\section{Some applications to critical point theory}

\subsection{A general minimization result}

A space of ultrafunctions has a "lot of compactness" since it is the $%
\Lambda $-limit of a net of finite dimensional spaces. Probably, the
simplest example to show this fact is the following theorem:

\begin{theorem}
\label{B}Let%
\begin{equation*}
J:V\left( \Omega \right) \rightarrow \mathbb{R}
\end{equation*}%
be an operator continuous and coercive on finite dimensional spaces. Then the
operator%
\begin{equation*}
\widetilde{J}:V_{\Lambda }\left( \Omega \right) \rightarrow \mathbb{R}^{\ast
}
\end{equation*}%
has a minimum point. If $J$ itself has a minimizer $u,$ then $u^{\ast }$ is
a minimizer of $\widetilde{J}.$
\end{theorem}

\begin{proof} Take $\lambda \in \mathfrak{L}$; since the operator 
\begin{equation*}
J|_{V_{\lambda }}:V_{\lambda }\left( \Omega \right) \rightarrow \mathbb{R}
\end{equation*}%
is continuous and coercive, it has a minimizer; namely%
\begin{equation*}
\exists u_{\lambda }\in V_{\lambda },\ \forall v\in V_{\lambda },\
J(u_{\lambda })\geq J(v).
\end{equation*}%
Now set%
\begin{equation*}
u_{\Lambda }=\ \lim_{\lambda \uparrow \Lambda }\ u_{\lambda }.
\end{equation*}
We will show that $u_{\Lambda }$ is minimizer of $\widetilde{J}.$ We apply
Th. \ref{billo} with 
\begin{equation*}
x\mathcal{R}Y:=\forall v\in Y,\ J(x)\geq J(v).
\end{equation*}%
Then, since $V_{\Lambda }\left( \Omega \right) =\lim_{\lambda \uparrow
\Lambda }\ V_{\lambda },$ the following relation holds: 
\begin{equation*}
\forall v\in V_{\Lambda }\left( \Omega \right) ,\ J^{\ast }(u_{\Lambda
})\geq J^{\ast }(v).
\end{equation*}

If $J$ itself has a minimizer $\bar{u},$ then $u_{\lambda }$ is eventually
equal to $\bar{u}$ and hence $u_{\Lambda }=\bar{u}^{\ast }.$\end{proof}

\bigskip

\textbf{Example}: Let us consider a classical problem of calculus of
variations: minimize the functional%
\begin{equation}
J(u)=\int F(x,u,\nabla u)dx  \label{xxx}
\end{equation}%
in the function space $\mathcal{C}_{0}^{1}(\Omega )=\mathcal{C}^{1}(\Omega
)\cap \mathcal{C}_{0}(\overline{\Omega }).$ Here we assume $\Omega $ to be
bounded so we do not have problems of summability.

It is well known that in general this problem has no solution even when $F$
is coercive and the infimum exists. However, if $F$ is convex and $%
\partial \Omega $ is sufficiently smooth, it is possible to find a minimizer
in a suitable Sobolev space (or in some "Sobolev type" space such as Orliz
spaces).

If $F$ is not convex it is not possible to find a minimizer, not even among
the generalized functions of "Sobolev" type, as the following example shows:%
\begin{equation}
\text{\textit{minimize}}\ \ \ J_{0}(u)=\int_{0}^{1}\left[ \left( \left\vert
\nabla u\right\vert ^{2}-1\right) ^{2}+|u|^{2}\right] dx\ \ \ \text{in}\ \ \ 
\mathcal{C}_{0}^{1}(0,1).  \label{lilla}
\end{equation}

It is not difficult to realize that any minimizing sequence $u_{n}$
converges uniformly to $0$ and that $J_{0}(u_{n})\rightarrow 0,$ but $%
J_{0}(0)>0$ for any $u\in \mathcal{C}_{0}^{1}(0,1)$ (and also for any $u\in
W_{0}^{1,4}(0,1)$).

On the contrary, it is possible to show that these problems have minimizers
in spaces of ultrafunctions; a natural space to work in is 
\begin{eqnarray*}
V_{0}^{1}(\Omega ) &=&\left\{ u\in V_{\Lambda }^{1}(\Omega )\ |\ u(x)=0\ \ 
\text{for\ every\ }x\in \partial \Omega ^{\ast }\right\} \\
&=&\left[ \mathcal{C}^{1}(\Omega )\cap \mathcal{C}_{0}(\overline{\Omega })%
\right] _{\Lambda }.
\end{eqnarray*}%
So our problem becomes%
\begin{equation*}
\underset{u\in V_{0}^{1}(\Omega )}{\min }J^{\ast }(u).
\end{equation*}

\begin{theorem}
Assume that $F$ is continuous and that 
\begin{equation}
F(x,u,\xi )\geq a(\xi )-M,  \label{glu}
\end{equation}%
where $a(\xi )\rightarrow +\infty $ as $\xi \rightarrow +\infty $ and $M$ is
a constant. Then, if $J$ is given by (\ref{xxx}),%
\begin{equation*}
\underset{u\in V_{0}^{1}(\Omega )}{\min }\widetilde{J}(u)
\end{equation*}%
exists.
\end{theorem}

\begin{proof} It is immediate to check that the assumptions of Th. \ref{B}
are verified.\end{proof}

\bigskip

In particular, this result applies to the functional (\ref{lilla}). It is
not difficult to show that, $\forall x\in (0,1)^{\ast },$ the minimizer $%
u_{\Lambda }(x)\sim 0$ and that $J_{0}(u_{\Lambda })$ is a positive
infinitesimal.

\subsection{Mountain pass theorem for ultrafunctions}

The Mountain Pass theorem of Ambrosetti and Rabinowitz \cite{AmRa1973} is a
well known theorem in Nonlinear Analysis with lots of applications. Next, we
will present one of the possible variants in the framework of ultrafunction:

\begin{theorem}
\label{MPT}(\textbf{Mountain pass theorem for ultrafunctions}) Let%
\begin{equation*}
J:V\rightarrow \mathbb{R}
\end{equation*}%
be an operator differentiable on the finite dimensional spaces. Assume the
following geometrical assumptions: $J(0)=0$, $0$ is an isolated minimum of $%
J $ and 
\begin{equation*}
\ \underset{u\in V_{\lambda };\ u\rightarrow \infty }{\lim \inf }J(u)\leq 0.
\end{equation*}%
Then $\widetilde{J}$ has a strictly positive critical value.
\end{theorem}

\begin{proof} For every $\lambda \in \mathfrak{L}$ the operator%
\begin{equation*}
J|_{V_{\lambda }}:V_{\lambda }\rightarrow \mathbb{R}
\end{equation*}%
has a mountain pass point $u_{\lambda }\ $(notice that $V_{\lambda }$ is
eventually $\neq \emptyset$). This fact is easy to prove since the set $%
A_{\lambda }^{+}=\left\{ u\in V_{\lambda }\ |\ J(u)\geq 0\right\} $ is
compact and hence PS\footnote{%
We recall that the Palais -Smale condition (which is abbreviated by PS) is a
basic tool for Critical Point theory in infinite dimensional spaces.} holds (we refer the reader who is not familiar with
this topic to the original article \cite{AmRa1973}).

Then, as usual, we set%
\begin{equation*}
u_{\Lambda }=\ \lim_{\lambda \uparrow \Lambda }\ u_{\lambda }
\end{equation*}%
and we want to prove that $u_{\Lambda }$ is a critical value. We have that, $%
\forall \lambda \in \mathfrak{L}$, 
\begin{equation*}
\forall v\in V_{\lambda },\ J^{\prime }(u_{\lambda })\left[ v\right] =0.
\end{equation*}%
By virtue of Th. \ref{billo}, we can take the $\Lambda $-limit in the above
formula and we get that 
\begin{equation*}
\forall v\in V_{\Lambda },\ \widetilde{J^{\prime }}(u_{\Lambda })\left[ v%
\right] =0.
\end{equation*}%
Hence $u_{\Lambda }$ is a critical point of $\widetilde{J}$ and, if $J$
itself has a mountain pass point $\bar{u}$, then $u_{\lambda }$ is identically equal to $\bar{u}$
and hence $u_{\Lambda }$=$\bar{u}^{\ast }.$\end{proof}

\bigskip

\textbf{Example: }Let us consider the functional

\begin{equation}
J(u)=\frac{1}{p}\int \left\vert \nabla u\right\vert ^{p}dx-\int F(x,u)\ dx,\
\ p\geq 2.  \label{pippo}
\end{equation}%
where $F(x,u)$ is a continuous function which satisfies the following
assumptions:

\begin{enumerate}
\item $F$ is differentiable with respect to $u;$

\item $\exists a,M>0,$ and $q_{1}>p$ such that $|u|<a\Rightarrow F(x,u)\leq
M\ \left\vert u\right\vert ^{q_{1}};\ $

\item $\exists b,M>0,\ $and $q_{2}>p\ $such that $|u|>a\ \Rightarrow
F(x,u)\geq b\left\vert u\right\vert ^{q_{2}}$.
\end{enumerate}

It is easy to check that the above requests on $F$ imply the assumptions of
Th. \ref{MPT}. However, the above requests are not sufficient to ensure PS and
hence, in general, there is not a Mountain Pass solution of the functional (%
\ref{pippo}) in any Sobolev space.

\subsection{Critical points in $V_{\Lambda }$ versus critical points in
Sobolev spaces}

Theorem \ref{MPT} states some facts about the critical points of $\widetilde{%
J}$; the next theorem will establish some relations between the critical
points of $\widetilde{J}$ in $V_{\Lambda }$ and the critical points of $J$
in $V.$

The first result in this direction is (almost) trivial:

\begin{theorem}
\label{cecilia}Let $V$ be a Banach space and let%
\begin{equation*}
J:V\rightarrow \mathbb{R}
\end{equation*}%
be a differentiable operator. Then if $u_{0}$ is a critical point of $J,$ $%
u_{0}^{\ast }$ is a critical point of $\widetilde{J}:V_{\Lambda }\rightarrow 
\mathbb{R}.$
\end{theorem}

\begin{proof} Let $u_{0}\in V$ be a critical point of $J.$ Take $\lambda
_{0}$ such that $u_{0}\in V_{\lambda _{0}}$. The set%
\begin{equation*}
\mathfrak{Q}(\lambda _{0})=\left\{ \lambda \in \mathfrak{L}\ |\ \lambda
\supset \lambda _{0}\right\}
\end{equation*}%
is qualified and we have that, $\forall \lambda \in \mathfrak{Q}(\lambda
_{0})$, 
\begin{equation*}
\forall v\in V_{\lambda },\ J^{\prime }(u_{0})\left[ v\right] =0.
\end{equation*}%
By virtue of Th. \ref{billo}, we can take the $\Lambda $-limit in the above
formula and we get that 
\begin{equation*}
\forall v\in V_{\Lambda },\ \widetilde{J^{\prime }}(u_{0}^{\ast })\left[ v%
\right] =0.
\end{equation*}
\end{proof}

\bigskip

The above theorem cannot be inverted in the sense that it is false that
every critical point of $\widetilde{J}$ corresponds to a critical point of $%
J.$ However, there are conditions which ensure the existence of critical
points of $J$ in $W$. More precisely the next theorem states that, under
suitable conditions, there is a critical point of $J$
"infinitely close" to any critical point of $\widetilde{J%
}$.
This theorem exploits a compactness condition which is a variant of the
usual Palais-Smale condition. Here, a variant of this condition is used just
to relate the critical points of $\widetilde{J}$ with the critical points of 
$J.$

\begin{definition}
\label{PSU}(\textbf{Palais-Smale condition for ultrafunctions (PSU))} Let $V$
be a Banach space. We say that the functional%
\begin{equation*}
J:V\rightarrow \mathbb{R}
\end{equation*}%
satisfies (PSU) in the interval $\left[ a,b\right] \subset \mathbb{R}$ if
every net $\left\{ u_{\lambda }\right\} _{\lambda \in \mathfrak{D}}$ ($%
\mathfrak{D}\subset \mathfrak{L}$) such that

\begin{itemize}
\item (A) $\forall \lambda \in \mathfrak{D,\ }J(u_{\lambda })\in \left[ a,b%
\right] ;$

\item (B) $\forall \lambda \in \mathfrak{D},\ \forall v\in V_{\lambda },\
dJ(u_{\lambda })\left[ v\right] =0;$

\end{itemize}

\noindent has a converging subnet in V.

\end{definition}

\begin{theorem}
\label{A}Assume the same framework and the same hypotheses of Th. \ref%
{cecilia}. Moreover, assume that $J$ satisfies (PSU) in the interval $\left[
a,b\right] .$ Then, if $\bar{u}$ is a critical point of $\widetilde{J}$ with 
$\widetilde{J}\left( \bar{u}\right) \in \left[ a,b\right] ^{\ast },$ there
exists a critical point $w$ of $J\in V$ such that%
\begin{equation*}
\left\Vert \bar{u}-w\right\Vert _{V^{\ast }}\sim 0.
\end{equation*}
\end{theorem}

\begin{remark}
Notice that, in the above theorem, it is possible that $\bar{u}\in V;$ then
in this case we have that $\bar{u}=w.$ Obviously, this facts always occur if 
$V$ is a Hilbert space and all the critical values of $J$ in $\left[ a,b%
\right] $ are not degenerate.
\end{remark}

\bigskip

\begin{proof} Proof of Th. \ref{A}: Since\textbf{\ }$\bar{u}\in V$ is a critical
point with critical value in $\left[ a,b\right] ,$ then%
\begin{equation*}
\bar{u}=\lim_{\lambda \uparrow \Lambda }\ u_{\lambda }
\end{equation*}%
where $\left\{ u_{\lambda }\right\} _{\lambda \in \mathfrak{L}}$ is a net
which satisfies (A) and (B) of (PSU). Thus $\left\{ u_{\lambda }\right\}
_{\lambda \in \mathfrak{L}}$ has a converging subnet $\left\{ u_{\lambda
}\right\} _{\lambda \in \mathfrak{A}}$ ($\mathfrak{A}\subset \mathfrak{L}$)
to a critical point $w\in V.$

Let $\ dist_{V}$ be the distance in $V.$ We set 
\begin{equation*}
\mathfrak{L}_{n}=\left\{ \lambda \in \mathfrak{L\ }|\ dist_{V}(u_{\lambda
},K_{a}^{b}\cap V_{\lambda })\leq \frac{1}{n}\right\} ,
\end{equation*}%
where $K_{a}^{b}$ is the set of critical points of $J$ with value in $\left[
a,b\right] .$

There are two possibilities:

\begin{itemize}
\item (A) $\exists \bar{n}\in \mathbb{N}$ such that the set $\mathfrak{L}_{%
\bar{n}}$ is not qualified (see sec. \ref{qs}), or

\item (B) $\forall n\in \mathbb{N}$, the set $\mathfrak{L}_{n}$ is qualified.
\end{itemize}

We will show that (A) cannot hold. We argue indirectly and suppose that (A)
holds. Then, by Prop. \ref{quaqua},4, the set $\mathfrak{Q}=\mathfrak{L}%
\backslash \mathfrak{L}_{\bar{n}}$ is qualified and since%
\begin{equation*}
\forall \lambda \in \mathfrak{Q},\ dist_{V}(u_{\lambda },K_{a}^{b}\cap
V_{\lambda })>\frac{1}{\bar{n}},
\end{equation*}%
we have that 
\begin{equation*}
dist_{V}^{\ast }(\bar{u},\left( K_{a}^{b}\right) ^{\ast }\cap V_{\Lambda })>\frac{1}{%
\bar{n}},
\end{equation*}%
and since $K_{a}^{b}\subset \left( K_{a}^{b}\right) ^{\ast }\cap V_{\Lambda },$ we have
that%
\begin{equation*}
\forall w\in K_{a}^{b},\ dist_{V_{\Lambda }}(\bar{u},w)>\frac{1}{\bar{n}},
\end{equation*}%
but this is not possible since $u_{\lambda }$ contains a subnet which
converges to a critical point.

Since (B) holds then, if we fix $n,$ we have that%
\begin{equation*}
\forall \lambda \in \mathfrak{L}_{n},\ dist_{V}(u_{\lambda },K_{a}^{b}\cap
V_{\lambda })<\frac{1}{n}.
\end{equation*}%
Then, taking the $\Lambda $-limit, we get that 
\begin{equation*}
dist_{V_{\Lambda }}(\bar{u},\left( K_{a}^{b}\right) ^{\ast }\cap V_{\Lambda })<\frac{1}{n}
\end{equation*}%
and, by the arbitrariness of $n,$ we get that $dist_{V_{\Lambda }}(\bar{u}%
,\left( K_{a}^{b}\right) ^{\ast }\cap V_{\Lambda })\sim 0$. In particular $\exists v\in
\left( K_{a}^{b}\right) ^{\ast }$ such that $\left\Vert \bar{u}-v\right\Vert
_{V_{\Lambda }}\sim 0.$ By (PSU) it follows that $K_{a}^{b}$ is compact so, by the
nonstandard characterization of compact sets, there exists $w\in K_{a}^{b}$
such that $\left\Vert w-v\right\Vert _{V}\sim 0.$ Concluding, $\left\Vert 
\bar{u}-w\right\Vert _{W}\sim 0.$\end{proof}

\section{Applications to boundary value problems for second order operators}

\subsection{The general procedure}

In this section we will describe the general procedure to deal with problems
of the type:%
\begin{equation*}
\text{\textit{Find}\ \ \ }u\in V(\Omega )\ \ \ \text{\textit{such that}}
\end{equation*}%
\begin{equation*}
A(u)=f,
\end{equation*}%
where $A:V(\Omega )\rightarrow W$ is a differential operator and $f\in W$.

The "typical" formulation of this problem in the framework of ultrafunction
is the following one:%
\begin{equation*}
\text{\textit{Find} \ }u\in V_{\Lambda }(\Omega )\ \ \text{\textit{such that}%
}
\end{equation*}%
\begin{equation}
\forall \varphi \in V_{\Lambda }(\Omega ),\ \int_{\Omega ^{\ast }}^{\ast
}A^{\ast }(u)\varphi dx=\int_{\Omega ^{\ast }}^{\ast }f^{\ast }\varphi dx. 
\tag{P}  \label{P}
\end{equation}%
Clearly this formulation is possible if $\int_{\Omega ^{\ast }}^{\ast
}f^{\ast }\varphi dx$ makes sense, namely if $W\subset V^{\prime }(\Omega )$
or if $f$ can be identified with a distribution $T_{f}.$ In this case, using
the definitions and the notation introduced in section \ref{effo}, problem
can be rewritten in the following way:%
\begin{equation*}
\text{\textit{Find}\ \ }u\in V_{\Lambda }(\Omega )\ \ \text{\textit{such that%
}}
\end{equation*}%
\begin{equation*}
\widetilde{A}(u)=\widetilde{f}.
\end{equation*}%
Following the general strategy in the theory of ultrafunction, Problem (P)
can be reduced to study the following approximate problems:%
\begin{equation*}
\text{\textit{Find} }u_{\lambda }\in V_{\lambda }(\Omega )\text{ \ \textit{%
such\ that}}
\end{equation*}%
\begin{equation*}
\forall \varphi \in V_{\lambda }(\Omega ),\ \int_{\Omega }A(u_{\lambda
})\varphi dx=\int_{\Omega }f\varphi dx.  
\end{equation*}%
The next steps consist in solving the approximate problems for
every $\lambda $ in a qualified set and in taking the $\Lambda $-limit.
Clearly, this strategy can be applied to a very large class of problems. In
the following sections we will see some of them with some details.

\subsection{The choice of the space $V(\Omega )$}

Let us consider the following abstract problem:%
\begin{equation*}
A(u)=f\ \ \text{in} \Omega \ +\ \text{boundary\ conditions,}
\end{equation*}%
where $\Omega $ is an open set in $\mathbb{R}^{N}$ and is $A$ is a
differential operator. Usually this kind of problems have a "natural space"
where to look for solutions; for example if $A$ is a second order
differential operator with sufficiently smooth coefficients, the \textit{%
natural space} where to look for solutions is $\mathcal{C}_{BC}^{2}(\Omega )$%
, namely the set of functions of class $\mathcal{C}^{2}$ which satisfy
suitable boundary conditions. However, many times the "natural space" is
inadequate to study the problem, since there is no solution in it. In
general the choice of the appropriate function space is part of the problem
itself. The appropriate function space is a space in which the problem is
well posed and (relatively) easy to be solved. For a very large class of
problems it is a Sobolev space. Some extra assumptions (such as the
regularity of $\partial \Omega $) might guarantee the existence of solutions
in the natural space.

The choice of the appropriate function space is somewhat arbitrary and it
might depend on the final goals. In the framework of ultrafunctions this
situation persists. However, in this case, there is a general rule: choose
the "natural space" $V(\Omega )$ and look for a generalized solution in $%
V_{\Lambda }(\Omega ).$ In the examples considered here we will follow this
general rule.

\subsection{Linear problems}

The most general linear operator of the second order is the following one: 
\begin{equation*}
Lu=\sum_{i,j}a_{ij}(x)\frac{\partial ^{2}u}{\partial x_{i}\partial x_{j}}%
+\sum_{i}b_{i}(x)\frac{\partial u}{\partial x_{i}}+c(x)u.
\end{equation*}%
So, let us consider the following problem:%
\begin{equation*}
\text{\textit{Find}\ \ }u\in \mathcal{C}_{BC}^{2}(\Omega )\ \ \text{\textit{%
such that}}
\end{equation*}%
\begin{equation*}
Lu=f.
\end{equation*}%
This problem, in the framework of ultrafunctions, becomes%
\begin{equation*}
\forall \varphi \in V_{\Lambda }(\Omega ),\ \int_{\Omega }L^{\ast }u\varphi
~dx=\int_{\Omega }f^{\ast }\varphi ~dx,
\end{equation*}%
where $V(\Omega )=\mathcal{C}_{BC}^{2}(\Omega ).$ Using the notation of
section \ref{effo}, we get:%
\begin{equation*}
\text{\textit{Find}\ \ \ }u\in V_{\Lambda }(\Omega )\ \ \ \text{\textit{such that}%
}
\end{equation*}%

\begin{equation*}\label{PL}
\widetilde{L}u=\widetilde{f}. 
\end{equation*}

In this case our problem becomes a finite dimensional problem from the space 
$V_{\lambda }$ to $V_{\lambda }$. Therefore the Fredholm alternative holds
for every $\lambda $ and hence, using the Transfer Principle, (Th. \ref%
{billo}), we get the following result:

\begin{theorem}
\label{LS}Either the problem%
\begin{equation*}
u\in V_{\Lambda }(\Omega ),\ \widetilde{L}u=0
\end{equation*}%
has infinitely many solutions or the problem%
\begin{equation*}
u\in V_{\Lambda }(\Omega ),\ \widetilde{L}u=\widetilde{f}
\end{equation*}%
has exactly one solution.
\end{theorem}

If $L$ restricted to $V_{\lambda }(\Omega )$ is symmetric then it has
exactly $\dim V_{\lambda }(\Omega )$ eigenvalues and the eigenfunctions form
an orthogonal basis. Taking the $\Lambda $-limit and using the transfer
principle we get the following result:

\begin{theorem}
If $L$ restricted to $V_{\lambda }(\Omega )$ is symmetric, the spectrum $\sigma (\widetilde{L})$ of $\widetilde{L}$  consists of a hyperfinite family $\left\{ \mu _{j}\right\}
_{j\in J}$ of hyperreal eigenvalues and there exists an orthonormal basis of
eigenfunctions $\left\{ e_{j}(x)\right\} _{j\in J}.$
\end{theorem}

In this case, if $0\notin \sigma (\widetilde{L}),$ the solution of problem (%
\ref{PL}) exists and it can be written explicitly as follows:%
\begin{equation*}
u(x)=\sum_{j\in J}\frac{1}{\mu _{j}}\left( \int_{\Omega }f^{\ast
}(y)e_{j}(y)dy\right) e_{j}(x).
\end{equation*}

\textbf{Example:} Take%
\begin{equation*}
Lu=\square u=\frac{\partial ^{2}u}{\partial t^{2}}-\frac{\partial ^{2}u}{%
\partial x^{2}}
\end{equation*}%
defined in $\Omega =(0,2\pi )\times \left( 0,L\right) $. We impose the
Dirichlet boundary condition with respect to $x$ and periodic boundary
conditions with respect to $t$, namely we take%
\begin{center}$
V_{BC}^{2}(\Omega )=$\\
$\left\{ v\in V_{\Lambda }^{2}(\Omega )\ |\ v(t,0)=v(t,L)=0,\ v(0,x)=v(2\pi ,L),\ v_t(0,x)=v_t(2\pi ,L)\right\}=$ \\
$\left[ \left\{ v\in \mathcal{C}^{2}(\Omega )\cap \mathcal{C}^{0}(%
\overline{\Omega })\ |\ v(t,0)=v(t,L)=0,\ v(0,x)=v(2\pi ,L),\ v_t(0,x)=v_t(2\pi ,L)\right\} \right]_{\Lambda }.$
\end{center}%
The eigenvalues of (the self-adjoint realization) of $L$ are given by $%
\left\{ (\frac{l\pi }{L})^{2}-k^{2}\right\} _{k,l\in \mathbb{N}}\ $and the
corresponding normalized eigenfunctions are $\left\{ \frac{1}{\pi }\sin (%
\frac{l\pi }{L}x)e^{ikt}\right\} _{k\in \mathbb{Z},l\in \mathbb{N}}.$

The eigenvalues$\ $and the eigenfunctions of $\widetilde{L}$ are formally
the same, but the indices $k,l$ range from $-$ $\kappa _{1}$ to $\kappa _{1}$
and from $1$ to $\kappa _{2}$ respectively where $\kappa _{1},\kappa _{2}\in 
\mathbb{N}^{\ast }$ are infinite numbers.

If $\frac{\pi }{L}$ is an irrational number then $0\notin \sigma (\widetilde{%
L})$ and hence, for every $f\in L^{1}(\Omega ),$ there is a unique solution $%
u_{\Lambda }$ of the following problem%
\begin{equation*}
\forall \varphi \in V_{\Lambda }(\Omega ),\ \int_{\Omega }\left( \frac{%
\partial ^{2}u}{\partial t^{2}}-\frac{\partial ^{2}u}{\partial x^{2}}\right)
\varphi dx=\int_{\Omega }f\varphi dx,
\end{equation*}%
which can be interpreted as a periodic solution of the D'Alembert equation.

$u_{\Lambda }$ can be written explicitly as follows:%
\begin{equation*}
u_{\Lambda }(t,x)=\frac{1}{\pi }\sum_{k=-\kappa _{1}}^{\kappa
_{1}}\sum_{l=1}^{\kappa _{2}}\frac{f_{k,l}}{(\frac{l\pi }{L})^{2}-k^{2}}\sin
(\frac{l\pi }{L}x)e^{ikt},  \label{somme}
\end{equation*}%
where%
\begin{equation*}
f_{k,l}=\frac{1}{\pi }\int_{\Omega }f(t,x)\sin (\frac{l\pi }{L}x)e^{ikt}dt\
dx.
\end{equation*}%
Notice that the hyperfinite sum (\ref{somme}) converges pointwise in the
field of hyperreal numbers while the corresponding real series, in general,
does not converge because of the presence of the "small denominators" $(%
\frac{l\pi }{L})^{2}-k^{2}.$

\subsection{A nonlinear problem}

Now, let us consider a typical case in nonlinear problems:

\begin{theorem}
\label{nl}Let $A:V(\Omega )\rightarrow W$, $W\subset V^{\prime }(\Omega ),$
be an operator such that for every finite dimensional space $V_{\lambda
}\subset V(\Omega )$ there exists $R_{\lambda }\in \mathbb{R}$ such that%
\begin{equation*}
u\in V_{\lambda };\ \left\Vert u\right\Vert _{\sharp }=R_{\lambda
}\Longrightarrow \left\langle A(u),u\right\rangle >0,  \label{gilda}
\end{equation*}%
where $\left\Vert \cdot \right\Vert _{\sharp }$ is any norm in $V(\Omega ).$
Then the equation 
\begin{equation*}
\widetilde{A}(u)=0  \label{nonno}
\end{equation*}%
has at least one solution $u_{\Lambda }\in V_{\Lambda }(\Omega ).$
\end{theorem}

\begin{proof} If we set 
\begin{equation*}
B_{\lambda }=\left\{ u\in V_{\lambda }|\ \left\Vert u\right\Vert _{\sharp
}\leq R_{\lambda }\right\}
\end{equation*}%
and if $A_{\lambda }:V_{\lambda }\rightarrow V_{\lambda }$ is the operator
defined by the following relation:%
\begin{equation*}
\forall v\in V_{\lambda },\ \left\langle A_{\lambda }(u),v\right\rangle
=\left\langle A(u),v\right\rangle
\end{equation*}%
then, by (\ref{gilda}), it follows that $\deg (A_{\lambda },B_{\lambda
},0)=1,\ $where $\deg (\cdot ,\cdot ,\cdot )$ denotes the topological
degree (see e.g. \cite{AmMa2003}). Hence, $\forall \lambda \in \mathfrak{L}$,%
\begin{equation*}
\exists u\in V_{\lambda },\forall v\in V_{\lambda },\ \left\langle
A_{\lambda }(u),v\right\rangle =0.
\end{equation*}%
As usual, taking the limit we get a solution $u_{\Lambda }\in V_{\Lambda
}(\Omega )$ of eq. (\ref{nonno}).\end{proof}

\bigskip

\textbf{Example 1:} Let $\Omega $ be an open bounded in $\mathbb{R}^{N}$ set
and let $a(\cdot ,\cdot ,\cdot ),b(\cdot ,\cdot ,\cdot ):\mathbb{R}%
^{N}\times \mathbb{R}\times \overline{\Omega }\rightarrow \mathbb{R}^{N}$ be
a continuous functions such that $\forall \xi \in \mathbb{R}^{N},\forall
s\in \mathbb{R},\forall x\in \overline{\Omega }$ we have%
\begin{equation*}
a(\xi ,s,x)\cdot \xi +b(\xi ,s,x)\geq \nu \left( |\xi |\right) ,
\label{lalla}
\end{equation*}%
where $\nu $ is a monotone function 
\begin{equation*}
\nu \left( t\right) \rightarrow +\infty \ \text{for\ }t\rightarrow +\infty .
\label{lalla1}
\end{equation*}%
We consider the following problem:%
\begin{equation*}
\text{\textit{Find\ }\ }u\in \mathcal{C}_{0}^{2}(\Omega )\ \ \text{\textit{%
s.t.}}
\end{equation*}%
\begin{equation*}
\nabla \cdot a(\nabla u,u,x)=b(\nabla u,u,x).  \label{5th}
\end{equation*}%
In the framework of ultrafunctions this problem becomes%
\begin{equation*}
\text{\textit{Find} \ }u\in V_{0}^{2}(\Omega )\ \ \text{\textit{such that}}
\end{equation*}%
\begin{equation*}
\forall \varphi \in V_{0}^{2}(\Omega ),\ \int_{\Omega }\nabla \cdot a(\nabla
u,u,x)\ \varphi \ dx=\int_{\Omega }b(\nabla u,u,x)\varphi dx.  \notag
\end{equation*}%
If we set%
\begin{equation*}
A(u)=-\nabla \cdot a(\nabla u,u,x)+b(\nabla u,u,x)
\end{equation*}%
it is not difficult to check that (\ref{lalla}) and (\ref{lalla1}) are
sufficient to guarantee the assumptions of Th. \ref{nl} (with a suitable
Orliz norm) and hence the existence of a solution of problem (\ref{5th}).
Problem (\ref{5th}) covers well known situations such as the case in which $A
$ is a maximal monotone operator but also very pathological cases. E.g., if
one takes%
\begin{equation*}
a(\nabla u,u,x)=(|\nabla u|^{p-2}-1)\nabla u;\ b(\nabla u,u,x)=-f(x)
\end{equation*}%
one gets the problem to find\ $u\in V_{0}^{2}(\Omega )\ $such that%
\begin{equation*}
\forall \varphi \in V_{0}^{2}(\Omega ),\ \int_{\Omega }\left( \Delta
_{p}u-\Delta u\right) \ \varphi \ dx=\int_{\Omega }f^{\ast }\varphi dx. 
\notag
\end{equation*}%
Since%
\begin{equation*}
\int_{\Omega }\left( \Delta _{p}u-\Delta u\right) \ u\ dx=\left\Vert
u\right\Vert _{W_{0}^{1,p}}^{p}-\left\Vert u\right\Vert _{H_{0}^{1}}^{2},
\end{equation*}%
it is easy to check that we have a priori bounds (but not the convergence)
in $W_{0}^{1,p}(\Omega )$ and it seems interesting to study the kind of
regularity of the solutions.

\section{Application to physical problems involving material points}

\subsection{The notion of material point}

The notion of material point is a basic tool in Mathematical Physics since
the times of Euler, who introduced it. Even if material points (probably) do
not exist, nevertheless they are very useful in the description of nature
and they simplify the models so that they can be treated by mathematical
tools. However, as new notions entered in Physics (such as the notion of
field), the use of material points led to situations which required new
mathematics. For example, in order to describe the electric field generated
by a charged point we need the notion of Dirac measure $\delta _{q}$, namely
this field satisfies the following equation:%
\begin{equation}
\Delta u=\delta _{q},  \label{one}
\end{equation}%
where $\Delta $ is the Laplace operator.

In this section, we will describe two simple problems whose modellization
requires NAM. Let $\Omega \subset \mathbb{R}^{2}$ be an open bounded set
which represents (an ideal) membrane. Suppose that a material point $P$ is
placed in $\Omega $, let it be free to move.

Suppose that the point has a unit weight and that the only forces acting on
it are the gravitational force and the reaction of the membrane. If $q\in
\Omega $ is the position of the point and $u(x)$ represents the profile of
the membrane, it follows that equation (\ref{one}) holds in $\Omega $ with
boundary condition $u=0$ on $\partial \Omega .$

The question is:

\begin{center}
\textbf{which is the point} $q_{0}\in \Omega $ \textbf{that the particle
will occupy?}\\[0pt]
\end{center}

The natural way to approach this problem would be the following: for every $%
q\in \Omega $, the energy of the system is given by the elastic energy plus
the gravitational energy, namely%
\begin{equation}
E(u,q)=\frac{1}{2}\int_{\Omega }|\nabla u(x)|^{2}dx+u(q).  \label{integ}
\end{equation}%
If the couple $(u_{0},q_{0})$ minimizes $E$ then $q_{0}$ is the equilibrium
point. For every $q\in \Omega $, let $u_{q}(x)$ be the configuration when $P$
is placed in $q,$ namely the solution of equation (\ref{one}). So the
equilibrium point $q_{0}$ is the the point in which the function

\begin{equation}
F(q)=E(u_{q},q)
\end{equation}%
has a minimum.

In the classical context, this "natural" approach cannot be applied; in fact 
$u_{q}(x)$ has a singularity at the point $q$ which makes $u(q)$ not well
defined and the integral in (\ref{integ}) to diverge. On the contrary, this
problem can be treated in NAM as we will show. In fact, since infinite
numbers are allowed, we will be able to find a minimum configuration for the
energy (\ref{integ}).

\subsection{Equilibrium position of a material point}

Let $\Omega \subseteq \mathbb{R}^{N}$ be an open bounded set; we want to
find a function $u$ defined in $\Omega $ (with $u=0\ $on $\partial \Omega $)
and a point $q\in \Omega $ which minimize the functional%
\begin{equation*}
E(u,q)=\frac{1}{2}\int_{\Omega }|\nabla u(x)|^{2}dx+u(q).
\end{equation*}%
It is well known that this problem has no solution in $\mathcal{C}%
_{0}^{2}\left( \Omega \right) $ and it makes no sense in the space of
distributions. On the contrary it is well defined and it has a solution in $%
V_{0}^{2}(\Omega )=\left[ \mathcal{C}^{2}(\Omega )\cap \mathcal{C}^{0}(%
\overline{\Omega })\right] _{\Lambda }$. More exactly, we have the following
result:

\begin{theorem}
For every point $q\in \Omega ^{\ast },$ the Dirichlet problem%
\begin{equation}
\left\{ 
\begin{array}{cc}
\widetilde{\Delta }u=\delta _{q} & \text{for}\ \ x\in \Omega ^{\ast }, \\ 
u(x)=0 & \text{for\ }x\in \partial \Omega ^{\ast },%
\end{array}%
\right.  \label{caterina}
\end{equation}%
has a unique solution $u_{q}\in V_{0}^{2}(\Omega )$ whose energy $%
E(u_{q},q)\in \mathbb{R}^{\ast }$ is an infinite number; moreover there
exists $q_{\Lambda }\in \Omega ^{\ast }$ such that 
\begin{equation*}
E(u_{q_{\Lambda }},q_{\Lambda })=\ \underset{q\in \Omega ^{\ast }}{\min }%
E(u_{q},q)=\min\limits_{q\in \Omega ^{\ast },\ u\in V_{0}^{2}(\Omega
)}E(u,q).
\end{equation*}
\end{theorem}

\begin{proof} The existence of a solution $u_{q}\in V_{0}^{2}(\Omega )$
follows from Th. \ref{LS}.\textbf{\ }Now we observe that%
\begin{equation*}
\min\limits_{q\in \Omega ^{\ast }}E(u_{q},q)=\min\limits_{q\in \Omega ^{\ast
},\ u\in V_{\Lambda }(\Omega )}E(u,q).
\end{equation*}%
Thus, it is sufficient to minimize $E(u,q).$

$E(u,q)$ has a minimizer $\left( u_{\lambda },q_{\lambda }\right) \in
V_{\lambda }(\Omega )\times \overline{\Omega }$ since $E(u,q)$ is continuous
and coercive on the finite dimensional closed set $V_{\lambda }(\Omega
)\times \overline{\Omega }.$ Now set

\begin{equation*}
q_{\Lambda }=\lim_{\lambda \uparrow \Lambda }\ q_{\lambda };\ \ u_{\Lambda
}=\lim_{\lambda \uparrow \Lambda }\ u_{\lambda }.
\end{equation*}%
Arguing as in Th. \ref{B}, we have that $\left( u_{\Lambda },q_{\Lambda
}\right) \in V_{0}^{2}(\Omega )\times \overline{\Omega }^{\ast }$ is a
minimizer of $E(u,q)$ and $u_{\Lambda }$ solves eq. (\ref{caterina}) with $%
q=q_{\Lambda },$ and hence, $u_{\Lambda }=u_{q_{\Lambda }}.$

It remains to show that $q_{\Lambda }\notin \partial \Omega ^{\ast }.$ By
Th. \ref{delta}, 4, $\delta _{q}(x)$ is well defined also for $q\in \partial
\Omega ^{\ast }$ and, in this case, we have that $\delta _{q}(x)=0.$ It is
well known that for any $q\in \Omega ,$ $E(u_{q},q)$ is a negative number.
Thus, also $E\left( u_{\Lambda },q_{\Lambda }\right) $ is a negative number
and hence $q_{\Lambda }\notin \partial \Omega ^{\ast }$ since for every $%
q\in \partial \Omega ^{\ast },$ $E(u_{q},q)=0$.\end{proof}

\subsection{Equilibrium position of a charged material point in a box}

A similar problem that can be studied with the same technique is the problem
of an electrically charged pointwise free particle in a box. Representing
the box with an open bounded set $\Omega \subseteq \mathbb{R}^{3}$ and
denoting by $u_{q}$ the electrical potential generated by the particle
placed in $q\in \Omega $ we have that $u_{q}$ satisfies the Dirichlet problem%
\begin{equation*}
\left\{ 
\begin{array}{cc}
u\in V_{0}^{2}(\Omega ), &  \\ 
\widetilde{\Delta }u=\delta _{q} & \text{for}\ \ x\in \Omega .%
\end{array}%
\right. 
\end{equation*}%
The equilibrium point would be the point $q_{0}\in \Omega ^{\ast }$ that
minimizes the electrostatic energy which is given by%
\begin{equation*}
E_{el}(q)=\frac{1}{2}\int_{\Omega }|\nabla u_{q}(x)|^{2}dx.
\end{equation*}%
Notice that 
\begin{equation*}
E_{el}(q)=\int_{\Omega }\delta _{q}(x)u_{q}(x)dx-\frac{1}{2}\int_{\Omega
}|\nabla u_{q}(x)|^{2}dx,
\end{equation*}%
namely, on the solution, the electrostatic energy is the opposite of the
energy of a membrane-like problem in $\mathbb{R}^{3}$. In order to solve
this problem we notice that, by Th. \ref{delta}.4, we have that, for all $%
q\in \partial \Omega ,$ $\delta _{q}=0.$ So $E_{el}(q)\geq 0$ and $%
E_{el}(q)=0$ if and only if $q\in \partial \Omega .$ So the following
theorem holds:

\begin{theorem}
For every point $q\in \overline{\Omega }^{\ast },$ the Dirichlet problem%
\begin{equation}
\left\{ 
\begin{array}{cc}
\widetilde{\Delta }u=\delta _{q} & \text{for}\ \ x\in \Omega ^{\ast } \\ 
u(x)=0 & \text{for\ }x\in \partial \Omega ^{\ast }%
\end{array}%
\right.
\end{equation}%
is well defined and it has a unique solution $u_{q}\in V_{\Lambda }(\Omega
). $

\begin{itemize}
\item $E_{el}(q)\ $is infinite if the distance between $q\ $and $\partial
\Omega ^{\ast }$ is larger than a positive real number;

\item $E_{el}(q)$ is positive but not infinite for some $q\in \Omega ^{\ast
} $ infinitely close to $\partial \Omega ^{\ast };$

\item $E_{el}(q)=0$ if and only if $q\in \partial \Omega ^{\ast }.$
\end{itemize}

In particular the stable equilibrium positions of the charged material
points belong to the boundary $\partial \Omega ^{\ast }.$
\end{theorem}

\medskip
Received xxxx 20xx; revised xxxx 20xx.
\medskip

\end{document}